\documentclass{article}
\usepackage{amssymb,amsmath,amsfonts}

\newtheorem{theorem}{Theorem}[section]

\newtheorem{corollary}[theorem]{Corollary}
\newtheorem{definition}[theorem]{Definition}

\numberwithin{equation}{section}
\newenvironment{proof}{\par\noindent{\bf Proof.}}{$\square$\par\bigskip}

\begin{document}

\title{\textbf{A note on the bi-periodic Fibonacci and Lucas matrix sequences}}
\author{\texttt{Arzu Coskun, Nazmiye Yilmaz} and \texttt{Necati Taskara}%
\thanks{%
e mail: \ \textit{arzucoskun58@gmail.com, nzyilmaz@selcuk.edu.tr,
ntaskara@selcuk.edu.tr }}}
\date{Department of Mathematics, Faculty of Science,\\
Selcuk University, Campus, 42075, Konya - Turkey }
\maketitle

\begin{abstract}
In this paper, we introduce the bi-periodic Lucas \textit{matrix sequence}
and present some fundamental properties of this generalized matrix sequence.
Moreover, we investigate the important relationships between the bi-periodic
Fibonacci and Lucas matrix sequences. We express that some behaviours of
bi-periodic Lucas numbers also can be obtained by considering properties of
this new matrix sequence. Finally, we say that the matrix sequences as Lucas, $k$-Lucas and Pell-Lucas are special cases of this
generalized matrix sequence.

{\textit{Keywords:}} bi-periodic Fibonacci matrix sequence, bi-periodic
Lucas matrix sequence, bi-periodic Lucas number, generalized Fibonacci matrix sequence, generating function.

{\textit{AMS Classification:}} 11B39, 15A24.
\end{abstract}

\section{Introduction and Preliminaries}
\qquad 
There are so many studies in the literature that concern about the special
number sequences such as Fibonacci, Lucas, Pell, Jacobsthal, Padovan and
Perrin (see, for example \cite{Bilgici, EdsonYayenie, Falcon, Horadam, Koshy, Ocalark, YazlikTaskara,Yilmazark}, and the references cited therein). Especially, the Fibonacci and Lucas numbers have attracted the attention of mathematicians
because of their intrinsic theory and applications.

Many authors have generalized Fibonacci and Lucas sequences in different
ways. For example, in \cite{Bilgici, EdsonYayenie}, the authors defined the bi-periodic
Fibonacci $\left\{ q_{n}\right\} _{n\in \mathbb{N}}$ sequence as
\begin{equation}
\text{ }q_{n}=\left\{ 
\begin{array}{c}
aq_{n-1}+q_{n-2},\ \ \text{if }n\text{ is even} \\ 
bq_{n-1}+q_{n-2},\ \ \text{if }n\text{ is odd}%
\end{array}
\right. ,  \label{1.0}
\end{equation}%
and the bi-periodic Lucas $\left\{ l_{n}\right\} _{n\in \mathbb{N}}$
sequence as in the form
\begin{equation}
l_{n}=\left\{ 
\begin{array}{c}
al_{n-1}+l_{n-2},\ \ \text{if }n\text{ is odd} \\ 
bl_{n-1}+l_{n-2},\ \ \text{if }n\text{ is even}
\end{array}%
\right. ,  \label{1.1}
\end{equation}
where $q_{0}=0,\ q_{1}=1,$ $l_{0}=2$, $l_{1}=a$ and $a, b$ are nonzero real numbers. Also, in \cite{Bilgici}, Bilgici gave some
relations between bi-periodic Fibonacci and Lucas numbers as in the
following:
\begin{equation}\label{iliski}
l_{n}=q_{n-1}+q_{n+1},
\end{equation}
\begin{equation}\label{1.2}
\left( ab+4\right) q_{n}=l_{n+1}+l_{n-1}.  
\end{equation}

On the other hand, the matrix sequences have taken so much interest for
different type of numbers (\cite{CivcivTurkmen, CivcivTurkmen1, CoskunTaskara, GulecTaskara, Ipekark, UsluUygun, Yazlikark, YilmazTaskara}). In \cite{CoskunTaskara}, the authors
defined bi-periodic Fibonacci matrix sequence and obtained $n$th general
term of this matrix sequence as
\begin{equation} \label{1.3}
\mathcal{F}_{n}\left( a,b\right) =\left( 
\begin{array}{cc}
\left( \frac{b}{a}\right) ^{\varepsilon (n)}q_{n+1} & \frac{b}{a}q_{n} \\ 
q_{n} & \left( \frac{b}{a}\right) ^{\varepsilon (n)}q_{n-1}
\end{array}%
\right) \,, 
\end{equation}
where
\begin{equation}\label{epsilon}
\varepsilon (n)=\left\{ 
\begin{array}{c}
1,\ n\text{ odd} \\ 
0,\ n\text{ even}
\end{array}
\right.
\end{equation}
In addition, the authors found the Binet formula of the bi-periodic Fibonacci
matrix sequence as in the following
\begin{equation}\label{fibobinet}
\mathcal{F}_{n}\left( a,b\right) =A_{1}\left( \alpha ^{n}-\beta ^{n}\right)
+B_{1}\left( \alpha ^{2\left\lfloor \frac{n}{2}\right\rfloor +2}-\beta
^{2\left\lfloor \frac{n}{2}\right\rfloor +2}\right) ,
\end{equation}
where $A_{1}=\dfrac{\left[ {\mathcal{F}}_{1}\left( a,b\right) -b\mathcal{F}%
_{0}\left( a,b\right) \right] ^{\varepsilon (n)}\left[ a{\mathcal{F}}%
_{1}\left( a,b\right) -\mathcal{F}_{0}\left( a,b\right) -ab\mathcal{F}%
_{0}\left( a,b\right) \right] ^{1-\varepsilon (n)}}{\left( ab\right)
^{\left\lfloor \frac{n}{2}\right\rfloor }\left( \alpha -\beta \right) }$,
$B_{1}=\dfrac{b^{\varepsilon (n)}\mathcal{F}_{0}\left( a,b\right) }{\left(
ab\right) ^{\left\lfloor \frac{n}{2}\right\rfloor +1}\left( \alpha -\beta
\right) }$.\\
\par
In the light of all these above material, the main goal of this paper is to investigate the relationships between the bi-periodic Fibonacci and bi-periodic Lucas matrix sequences. To do that, firstly, we define the bi-periodic Lucas matrix sequences, because it is
worth to study a new matrix sequence related to less known numbers. Then, it
will be given the generating function, Binet formula and summation formulas
for this new matrix sequence. By using the results in Sections 2, we have a
great opportunity to obtain some new properties over this matrix sequence.

\section{The matrix sequence of bi-periodic Lucas numbers}
\qquad 
In this section, we mainly focus on the matrix sequence of the bi-periodic Lucas numbers. In fact, we present the some properties and Binet
formula of this matrix sequence. Also, we investigate various summations of this matrix sequence.

Now, we firstly define the bi-periodic Lucas matrix sequence as in the
following.
\begin{definition}
\label{tan} For $n\in \mathbb{N}$, the bi-periodic Lucas matrix sequence $\left( \mathcal{L}_{n}\left( a,b\right) \right) $ is defined by
\begin{equation}
\mathcal{L}_{n}\left( a,b\right) =\left\{ 
\begin{array}{c}
a\mathcal{L}_{n-1}\left( a,b\right) +\mathcal{L}_{n-2}\left( a,b\right) 
\text{, }\ n\text{ odd} \\ 
b\mathcal{L}_{n-1}\left( a,b\right) +\mathcal{L}_{n-2}\left( a,b\right) 
\text{, }\ n\text{ even}\ 
\end{array}%
\right. \,,  \label{2.0}
\end{equation}
with initial conditions
$\mathcal{L}_{0}\left( a,b\right) =\left(
\begin{array}{cc}
a & 2 \\ 
2\frac{a}{b} & -a%
\end{array}%
\right), {\mathcal{L}}_{1}\left( a,b\right) =\left(
\begin{array}{cc}
a^{2}+2\frac{a}{b} & a \\ 
\frac{a^{2}}{b} & 2\frac{a}{b}
\end{array}
\right)$ and $a, b$ are nonzero real numbers.
\end{definition}

In the following theorem, we give the $n$th general term of the matrix
sequence in (\ref{2.0}) via the bi-periodic Lucas numbers.

\begin{theorem}
\label{genel} For any integer $n\geq 0,$ we have the matrix sequence
\begin{equation}
\mathcal{L}_{n}\left( a,b\right) =\left( 
\begin{array}{cc}
\left( \frac{a}{b}\right) ^{\varepsilon (n)}l_{n+1} & l_{n} \\ 
\frac{a}{b}l_{n} & \left( \frac{a}{b}\right) ^{\varepsilon (n)}l_{n-1}
\end{array}
\right)\,,   \label{2.2}
\end{equation}
where $\varepsilon (n)$ is as in the equation (\ref{epsilon}).
\end{theorem}

\begin{proof}
The proof can be seen by using the induction method and the Equation (\ref{2.0}).
\end{proof}

In \cite{Bilgici}, the author obtained the Cassini identity for the bi-periodic Lucas
numbers. As a consequence of Theorem \ref{genel}, in the following corollary, we
rewrite this identity with a different approximation. 
\begin{corollary}
The following equalities are valid for all positive integers:
\begin{itemize}
\item Let $\mathcal{L}_{n}\left( a,b\right) $ be as in (\ref{2.2}). Then
\begin{equation}
\det (\mathcal{L}_{n}\left( a,b\right) )=\left( ab+4\right) \left( -\frac{a}{
b}\right) ^{1+\varepsilon (n)}.
\end{equation}\label{2.3}
\item Cassini identity can also be
obtained using bi-periodic Lucas matrix sequence. That is, by using Theorem 
\ref{genel} and the Equation (\ref{2.3}), we can write
\begin{equation*}
\left( \frac{b}{a}\right) ^{\varepsilon (n+1)}l_{n+1}l_{n-1}-\left( \frac{b}{%
a}\right) ^{\varepsilon (n)}l_{n}^{2}=\left( ab+4\right) \left( -1\right)
^{n+1}.
\end{equation*}
\end{itemize}
\end{corollary}

\begin{theorem}\label{oz1}
For every $n\in \mathbb{N}$, the following statements are true:
\begin{itemize}
\item[$\boldsymbol{(i)}$] \ \ $\mathcal{L}_{n-1}\left( a,b\right) +\mathcal{L}_{n+1}\left(
a,b\right) =\frac{a}{b}\left( ab+4\right) \mathcal{F}_{n}\left( a,b\right) ,$

\item[$\boldsymbol{(ii)}$] \ \ $\mathcal{F}_{n-1}\left( a,b\right) +\mathcal{F}%
_{n+1}\left( a,b\right) =\frac{b}{a}\mathcal{L}_{n}\left( a,b\right) .$
\end{itemize}
\end{theorem}
\begin{proof}
We will only prove $(i)$ since the proof of the other equality is similar to $(i)$. If we take  $T$ instead of $\mathcal{L}_{n-1}\left( a,b\right) +\mathcal{L}_{n+1}\left( a,b\right)$,  by using the Theorem \ref{genel}, we can write
\begin{eqnarray*}
T&=&\left( 
\begin{array}{cc}
\left( \frac{a}{b}\right) ^{1-\varepsilon (n)}l_{n} & l_{n-1} \\ 
\frac{a}{b}l_{n-1} & \left( \frac{a}{b}\right) ^{1-\varepsilon (n)}l_{n-2}%
\end{array}
\right) +\left( 
\begin{array}{cc}
\left( \frac{a}{b}\right) ^{1-\varepsilon (n)}l_{n+2} & l_{n+1} \\ 
\frac{a}{b}l_{n+1} & \left( \frac{a}{b}\right) ^{1-\varepsilon (n)}l_{n}%
\end{array}
\right) \\
&=&\left( 
\begin{array}{cc}
\left( \frac{a}{b}\right) ^{1-\varepsilon (n)}\left( l_{n+2}+l_{n}\right) & 
\left( l_{n+1}+l_{n-1}\right) \\ 
\frac{a}{b}\left( l_{n+1}+l_{n-1}\right) & \left( \frac{a}{b}\right)
^{1-\varepsilon (n)}\left( l_{n}+l_{n-2}\right)%
\end{array}
\right)\,.
\end{eqnarray*}
From the Equation (\ref{1.2}), we get 
\begin{eqnarray*}
T&=&\left( 
\begin{array}{cc}
\left( \frac{a}{b}\right) ^{1-\varepsilon (n)}\left( ab+4\right) q_{n+1} & 
\left( ab+4\right) q_{n} \\ 
\frac{a}{b}\left( ab+4\right) q_{n} & \left( \frac{a}{b}\right)
^{1-\varepsilon (n)}\left( ab+4\right) q_{n-1}%
\end{array}%
\right) \\
&=&\frac{a}{b}\left( ab+4\right) \left( 
\begin{array}{cc}
\left( \frac{b}{a}\right) ^{\varepsilon (n)}q_{n+1} & \frac{b}{a}q_{n} \\ 
q_{n} & \left( \frac{b}{a}\right) ^{\varepsilon (n)}q_{n-1}%
\end{array}
\right)
\end{eqnarray*}
which is desired.
\end{proof}

\begin{theorem}\label{binet}
For every $n\in \mathbb{N}$, we write the Binet formula
for the bi-periodic Lucas matrix sequence as the form 
\begin{equation*}
\mathcal{L}_{n}\left( a,b\right) =A\alpha ^{n}-B\beta ^{n}
\end{equation*}
where 
$
A =\dfrac{b\mathcal{L}_{1}\left( a,b\right) +\alpha \mathcal{L}_{0}\left(
a,b\right) -ab\mathcal{L}_{0}\left( a,b\right) }{b^{\varepsilon (n)}\left(
ab\right) ^{\left\lfloor \frac{n}{2}\right\rfloor }\left( \alpha -\beta
\right) }$, $B =\dfrac{b\mathcal{L}_{1}\left( a,b\right) +\beta \mathcal{L}_{0}\left(
a,b\right) -ab\mathcal{L}_{0}\left( a,b\right) }{b^{\varepsilon (n)}\left(
ab\right) ^{\left\lfloor \frac{n}{2}\right\rfloor }\left( \alpha -\beta
\right) }$
such that $\alpha =\frac{ab+\sqrt{a^{2}b^{2}+4ab}}{2}, \beta =\frac{ab-\sqrt{%
a^{2}b^{2}+4ab}}{2}$.
\end{theorem}

\begin{proof}
It is easily found by benefitting the condition $(ii)$ of Theorem \ref{oz1} and the Equation (\ref{fibobinet}).
\end{proof}

\begin{theorem}\label{top-ure}
The following equalities are hold:
\begin{itemize}
\item[$\boldsymbol{(i)}$] The generating function for bi-periodic Lucas matrix sequence is
\begin{equation*}
\sum\limits_{i=0}^{\infty }\mathcal{L}_{i}\left( a,b\right) x^{i}=\dfrac{1}{%
1-\left( ab+2\right) x^{2}+x^{4}}\left( 
\begin{array}{cc}
A_{2} & B_{2} \\ 
\frac{a}{b}B_{2} & C_{2}%
\end{array}
\right)\,,
\end{equation*}
where $A_{2}=a+\left( a^{2}+2\frac{a}{b}\right) x+ax^{2}-2\frac{a}{b}x^{3},$ 
$B_{2}=2+ax-\left( ab+2\right) x^{2}+ax^{3}$ and $C_{2}=-a+2\frac{a}{b}%
x+(3+ab)ax^{2}-\left( a^{2}+2\frac{a}{b}\right) x^{3}$.

\item[$\boldsymbol{(ii)}$] For $k\geq 0$, there exist
\begin{equation*}
\sum\limits_{k=0}^{n}\mathcal{L}_{k}\left( a,b\right) x^{-k}=\frac{1}{%
1-(ab+2)x^{2}+x^{4}}\left\{ 
\begin{array}{c}
\dfrac{\mathcal{L}_{n-1}\left( a,b\right) }{x^{n-1}}-\dfrac{\mathcal{L}%
_{n+1}\left( a,b\right) }{x^{n-3}}\\
+\dfrac{\mathcal{L}_{n}\left( a,b\right) }{x^{n}}-\dfrac{\mathcal{L}%
_{n+2}\left( a,b\right) }{x^{n+2}}\\
+x^{4}\mathcal{L}_{0}\left( a,b\right) +x^{3}\mathcal{L}_{1}\left( a,b\right)
\\
-x^{2}\left[ \left( ab+1\right) \mathcal{L}_{0}\left( a,b\right) -b\mathcal{L%
}_{1}\left( a,b\right) \right]\\
-x\left( \mathcal{L}_{1}\left( a,b\right) -a\mathcal{L}_{0}\left( a,b\right)
\right)
\end{array}
\right\}\,.
\end{equation*}

\item[$\boldsymbol{(iii)}$] For $k>0$, we have
\begin{equation*}
\sum\limits_{k=0}^{\infty }\mathcal{L}_{k}\left( a,b\right) x^{-k}=\frac{x}{%
1-(ab+2)x^{2}+x^{4}}\left(
\begin{array}{cc}
D & E \\ 
\frac{a}{b}E & F
\end{array}
\right)\,,
\end{equation*}
where $D=ax^{3}+\left( a^{2}+2\frac{a}{b}\right) x^{2}-ax+2\frac{a}{b},$ $%
E=2x^{3}+ax^{2}+\left( ab+2\right) x+a$ and $F=-ax^{3}+2\frac{a}{b}%
x^{2}-\left( a^{2}b+3a\right) x+a^{2}+2\frac{a}{b}$.

\item[$\boldsymbol{(iv)}$] For $k\geq 0$, the summation of the bi-periodic Lucas matrix sequence is
\begin{equation*}
\sum\limits_{k=0}^{n-1}\mathcal{L}_{k}\left( a,b\right) =\frac{1}{ab}\left\{ 
\begin{array}{c}
b^{\varepsilon \left( n\right) }a^{1-\varepsilon \left( n\right) }\mathcal{L}%
_{n}\left( a,b\right) +b^{1-\varepsilon \left( n\right) }a^{\varepsilon
\left( n\right) }\mathcal{L}_{n-1}\left( a,b\right) \\ 
-b\mathcal{L}_{1}\left( a,b\right) +ab\mathcal{L}_{0}\left( a,b\right) -a%
\mathcal{L}_{0}\left( a,b\right)%
\end{array}
\right\}\,.
\end{equation*}
\end{itemize}
\end{theorem}

\begin{proof}
We establish just condition $(i)$ . The proofs of $(ii),(iii)$ and $(iv)$ can be done by taking account Binet formula of this matrix sequence.
\begin{itemize}
\item[$\boldsymbol{(i)}$] Assume that $G(x)$ is the generating function for the sequence $
\left\{ \mathcal{L}_{n}\left( a,b\right) \right\} _{n\in \mathbb{N}}$. Then,
we have
\begin{equation*}
G\left( x\right) =\sum\limits_{i=0}^{\infty }\mathcal{L}_{i}\left(
a,b\right) x^{i}=\mathcal{L}_{0}\left( a,b\right) +\mathcal{L}_{1}\left(
a,b\right) x+\sum\limits_{i=2}^{\infty }\mathcal{L}_{i}\left( a,b\right)
x^{i}.
\end{equation*}
Thus, we can write
\begin{eqnarray*}
\left( 1-bx-x^{2}\right) G\left( x\right)  &=&\mathcal{L}_{0}\left(
a,b\right) +x\left( \mathcal{L}_{1}\left( a,b\right) -b\mathcal{L}_{0}\left(
a,b\right) \right)  \\
&&+\sum\limits_{i=2}^{\infty }\left( \mathcal{L}_{i}\left( a,b\right) -b%
\mathcal{L}_{i-1}\left( a,b\right) -\mathcal{L}_{i-2}\left( a,b\right)
\right) x^{i}.
\end{eqnarray*}

Since $\mathcal{L}_{2i}\left( a,b\right) =b\mathcal{L}_{2i-1}\left(
a,b\right) +\mathcal{L}_{2i-2}\left( a,b\right) $, we get
\begin{eqnarray*}
\left( 1-bx-x^{2}\right) G\left( x\right) &=&\mathcal{L}_{0}\left(
a,b\right) +x\left( \mathcal{L}_{1}\left( a,b\right) -b\mathcal{L}_{0}\left(
a,b\right) \right) \\
&&+\sum\limits_{i=1}^{\infty }\left( \mathcal{L}_{2i+1}\left( a,b\right) -b%
\mathcal{L}_{2i}\left( a,b\right) -\mathcal{L}_{2i-1}\left( a,b\right)
\right) x^{2i+1} \\
&=&\mathcal{L}_{0}\left( a,b\right) +x\left( \mathcal{L}_{1}\left(
a,b\right) -b\mathcal{L}_{0}\left( a,b\right) \right) \\
&&+\left( a-b\right) x\sum\limits_{i=1}^{\infty }\mathcal{L}_{2i}\left(
a,b\right) x^{2i}.
\end{eqnarray*}

Now, let
\begin{equation*}
g\left( x\right) =\sum\limits_{i=1}^{\infty }\mathcal{L}_{2i}\left(
a,b\right) x^{2i}.
\end{equation*}

Since
\begin{eqnarray*}
\mathcal{L}_{2i}\left( a,b\right) &=&b\mathcal{L}_{2i-1}\left( a,b\right) +%
\mathcal{L}_{2i-2}\left( a,b\right) \\
&=&(ab+1)\mathcal{L}_{2i-2}\left( a,b\right) +b\mathcal{L}_{2i-3}\left(
a,b\right) \\
&=&(ab+2)\mathcal{L}_{2i-2}\left( a,b\right) -\mathcal{L}_{2i-4}\left(
a,b\right) ,
\end{eqnarray*}

we have
\begin{eqnarray*}
\left( 1-\left( ab+2\right) x^{2}+x^{4}\right) g\left( x\right)  &=&\mathcal{%
L}_{2}\left( a,b\right) x^{2}+\mathcal{L}_{4}\left( a,b\right) x^{4}-(ab+2)%
\mathcal{L}_{2}\left( a,b\right) x^{4} \\
&&+\sum\limits_{i=3}^{\infty }\left[ \mathcal{L}_{2i}\left( a,b\right)
-(ab+2)\mathcal{L}_{2i-2}\left( a,b\right) +\mathcal{L}_{2i-4}\left(
a,b\right) \right] x^{2i}.
\end{eqnarray*}

Therefore,
\begin{eqnarray*}
g\left( x\right) &=&\frac{\mathcal{L}_{2}\left( a,b\right) x^{2}+\mathcal{L}%
_{4}\left( a,b\right) x^{4}-(ab+2)\mathcal{L}_{2}\left( a,b\right) x^{4}}{%
1-\left( ab+2\right) x^{2}+x^{4}} \\
&=&\frac{(\mathcal{L}_{0}\left( a,b\right) +b\mathcal{L}_{1}\left(
a,b\right) )x^{2}-\mathcal{L}_{0}\left( a,b\right) x^{4}}{1-\left(
ab+2\right) x^{2}+x^{4}}
\end{eqnarray*}%
and as a result, we get
\begin{equation*}
G\left( x\right) =\frac{1}{1-\left( ab+2\right) x^{2}+x^{4}}\left\{
\begin{array}{c}
\mathcal{L}_{0}\left( a,b\right) +x\mathcal{L}_{1}\left( a,b\right)\\
+x^{2}\left( b\mathcal{L}_{1}\left( a,b\right) -\mathcal{L}_{0}\left(
a,b\right) -ab\mathcal{L}_{0}\left( a,b\right) \right) \\ 
+x^{3}\left( a\mathcal{L}_{0}\left( a,b\right) -\mathcal{L}_{1}\left(
a,b\right) \right)%
\end{array}
\right\}\,.
\end{equation*}%
where $\mathcal{L}_{0}\left( a,b\right) =\left( 
\begin{array}{cc}
a & 2 \\ 
2\frac{a}{b} & -a%
\end{array}%
\right) ,\mathcal{L}_{1}\left( a,b\right) =\left( 
\begin{array}{cc}
a^{2}+2\frac{a}{b} & a \\ 
\frac{a^{2}}{b} & 2\frac{a}{b}%
\end{array}
\right)$. Thus, the desired expression is obtained.
\end{itemize}
\end{proof}

\section{Relationships between the bi-periodic Fibonacci and Lucas matrix sequences}
\qquad
The following theorem express that there always exist some interpasses
between the bi-periodic Fibonacci and Lucas matrix sequences.

\begin{theorem}
\label{oz2} For the matrix sequences $\left( \mathcal{F}_{n}\left(
a,b\right) \right) _{n\in \mathbb{N}}$ and $\left( \mathcal{L}_{n}\left( a,b\right) \right) _{n\in \mathbb{N}}$, the following equalities are satisfied:
\begin{itemize}
\item[$\boldsymbol{(i)}$] $\mathcal{L}_{0}\left( a,b\right) \mathcal{F}_{n}\left(
a,b\right) =\left( \frac{b}{a}\right) ^{\varepsilon \left( n\right) }%
\mathcal{L}_{n}\left( a,b\right) =\left( \frac{a}{b}\right) ^{\varepsilon
\left( n+1\right) }\left( \mathcal{F}_{n-1}\left( a,b\right) +\mathcal{F}%
_{n+1}\left( a,b\right) \right)$,

\item[$\boldsymbol{(ii)}$]$\mathcal{F}_{n}\left( a,b\right) \mathcal{L}_{0}\left(
a,b\right) =\mathcal{L}_{0}\left( a,b\right) \mathcal{F}_{n}\left(
a,b\right) =\left( \frac{b}{a}\right) ^{\varepsilon \left( n\right) }%
\mathcal{L}_{n}\left( a,b\right)$,

\item[$\boldsymbol{(iii)}$] \ $\mathcal{F}_{1}\left( a,b\right) \mathcal{L}_{n}\left(
a,b\right) =\left( \frac{b}{a}\right) ^{\varepsilon \left( n\right) }\left( 
\mathcal{F}_{n+2}\left( a,b\right) +\mathcal{F}_{n}\left( a,b\right) \right)
=\left( \frac{b}{a}\right) ^{\varepsilon \left( n+1\right) }\mathcal{L}%
_{n+1}\left( a,b\right)$,

\item[$\boldsymbol{(iv)}$] \ $\ \mathcal{L}_{n}\left( a,b\right) \mathcal{F}_{1}\left(
a,b\right) =\mathcal{F}_{1}\left( a,b\right) \mathcal{L}_{n}\left(
a,b\right) =\left( \frac{b}{a}\right) ^{\varepsilon \left( n+1\right) }%
\mathcal{L}_{n+1}\left( a,b\right)$.
\end{itemize}
\end{theorem}
\begin{proof}
From the Equations (\ref{iliski}), (\ref{1.2}), (\ref{1.3}) and (\ref{2.2}), desired expressions are obtained.
\end{proof}

\begin{theorem}\label{oz3}
For $m,n\in\mathbb{N}$, we have
\begin{itemize}
\item[$\boldsymbol{(i)}$] $\mathcal{F}_{m}\left( a,b\right) \mathcal{F}_{n}\left(
a,b\right) =\mathcal{F}_{n}\left( a,b\right) \mathcal{F}_{m}\left(
a,b\right)=\left( \frac{b}{a}\right) ^{\varepsilon \left( mn\right) }
\mathcal{F}_{m+n}\left( a,b\right)$,

\item[$\boldsymbol{(ii)}$] \ \ $\mathcal{F}_{m}\left( a,b\right) \mathcal{L}_{n}\left(
a,b\right) =\mathcal{L}_{n}\left( a,b\right) \mathcal{F}_{m}\left(
a,b\right) =\left( \frac{b}{a}\right) ^{\varepsilon \left( m\right)
\varepsilon \left( n+1\right) }\mathcal{L}_{m+n}\left( a,b\right)$,

\item[$\boldsymbol{(iii)}$] $\mathcal{L}_{m}\left( a,b\right) \mathcal{L}_{n}\left(
a,b\right) =\mathcal{L}_{n}\left( a,b\right) \mathcal{L}_{m}\left(
a,b\right)  =\left( \frac{a%
}{b}\right) ^{2-\left[ \varepsilon \left( m+1\right) \varepsilon \left(
n+1\right) \right] }\left( ab+4\right) \mathcal{F}_{m+n}\left( a,b\right)$.
\end{itemize}
\end{theorem}

\begin{proof}
\begin{itemize}
\item[$\boldsymbol{(i)}$] Here, we will just show the truthness of the equality $\mathcal{F}_{m}\left( a,b\right) \mathcal{F}_{n}\left(
a,b\right) =\left( \frac{b}{a}\right) ^{\varepsilon \left( mn\right) }%
\mathcal{F}_{m+n}\left( a,b\right) $ since the other can be done similarly. From the equation (\ref{fibobinet}), we can write
\begin{equation*}
\mathcal{F}_{n}\left( a,b\right) =\left\{ 
\begin{array}{c}
\frac{1}{\alpha -\beta }\left\{ \frac{a{\mathcal{F}}_{1}\left( a,b\right)
+\alpha \mathcal{F}_{0}\left( a,b\right) -ab\mathcal{F}_{0}\left( a,b\right) 
}{\left( ab\right) ^{\frac{n}{2}}}\alpha ^{n}-\frac{a{\mathcal{F}}_{1}\left(
a,b\right) +\beta \mathcal{F}_{0}\left( a,b\right) -ab\mathcal{F}_{0}\left(
a,b\right) }{\left( ab\right) ^{\frac{n}{2}}}\beta ^{n}\right\},\text{ \ }n
\text{\ even} \\ 
\frac{1}{\alpha -\beta }\left\{ \frac{\alpha \mathcal{F}_{1}\left(
a,b\right) +b\mathcal{F}_{0}\left( a,b\right) }{\left( ab\right) ^{\frac{n-1%
}{2}}}\alpha ^{n-1}-\frac{\beta \mathcal{F}_{1}\left( a,b\right) +b\mathcal{F%
}_{0}\left( a,b\right) }{\left( ab\right) ^{\frac{n-1}{2}}}\beta
^{n-1}\right\},\text{ \ }n\text{\ odd }
\end{array}%
\right. .
\end{equation*}

Let $K=a{\mathcal{F}}_{1}\left( a,b\right) +\alpha \mathcal{F}%
_{0}\left( a,b\right) -ab\mathcal{F}_{0}\left( a,b\right),\ L=a{\mathcal{F}}%
_{1}\left( a,b\right) +\beta \mathcal{F}_{0}\left( a,b\right) -ab\mathcal{F}%
_{0}\left( a,b\right) ,\ M=\alpha \mathcal{F}_{1}\left( a,b\right) +b\mathcal{F%
}_{0}\left( a,b\right) $ and $N=\beta \mathcal{F}_{1}\left( a,b\right) +b%
\mathcal{F}_{0}\left( a,b\right)$. Then, we have $K^{2}=\left( \alpha -\beta
\right)K,\ L^{2}=-\left( \alpha -\beta \right)L,\ M^{2}=\frac{\alpha ^{2}%
}{a^{2}}\left( \alpha -\beta \right) K,\\ N^{2}=-\frac{\beta ^{2}}{a^{2}}\left(
\alpha -\beta \right) L,\ KM=\left( \alpha -\beta \right)M,\ LN=-\left( \alpha -\beta \right) N,\ KL=KN=MN=LM=0$. Consequently, we get
\begin{equation*}
\mathcal{F}_{m}\left( a,b\right) \mathcal{F}_{n}\left( a,b\right) =\left\{ 
\begin{array}{c}
\mathcal{F}_{m+n},\ \ \ \ \ \ \ \ \  \ \ \ \ \ \  \ m,n \ \ \text{even} \\ 
\frac{b}{a}\mathcal{F}_{m+n}, \ \ \ \ \ \ \ \ \  \ \ \ \ \ m,n \ \ \text{odd} \\ 
\mathcal{F}_{m+n},\ m\ \text{even(odd)}, n\ \text{odd(even)}
\end{array}%
\right. 
\end{equation*}
which is desired.

\item[$\boldsymbol{(ii)}$] Here, we will just show the truthness of the equality $\mathcal{%
F}_{m}\left( a,b\right) \mathcal{L}_{n}\left( a,b\right) =\left( \frac{b}{a}%
\right) ^{\varepsilon \left( m\right) \varepsilon \left( n+1\right) }%
\mathcal{L}_{m+n}\left( a,b\right) $ since the other can be done similarly.
Now, by the condition $(ii)$ of Theorem \ref{oz2} and $(i)$, we write
\begin{eqnarray*}
\mathcal{F}_{m}\left( a,b\right) \mathcal{L}_{n}\left( a,b\right)& =&\left( 
\frac{a}{b}\right) ^{\varepsilon (n)}\mathcal{F}_{m}\left( a,b\right) 
\mathcal{F}_{n}\left( a,b\right) \mathcal{L}_{0}\left( a,b\right) \\
&=&\left( \frac{a}{b}\right) ^{\varepsilon (n)}\left( \frac{b}{a}\right)
^{\varepsilon \left( mn\right) }\mathcal{F}_{m+n}\left( a,b\right) \mathcal{L%
}_{0}\left( a,b\right)\\
& =&\left( \frac{a}{b}\right) ^{\varepsilon
(n)-\varepsilon \left( mn\right) -\varepsilon \left( m+n\right) }\mathcal{L}%
_{m+n}\left( a,b\right) \\
&=&\left( \frac{b}{a}\right) ^{\varepsilon \left( m\right) \varepsilon
\left( n+1\right) }\mathcal{L}_{m+n}\left( a,b\right)\,.
\end{eqnarray*}

\item[$\boldsymbol{(iii)}$] We will just show the $\mathcal{L}_{m}\left( a,b\right) \mathcal{L}_{n}\left(
a,b\right) =\left( \frac{a%
}{b}\right) ^{2-\left[ \varepsilon \left( m+1\right) \varepsilon \left(
n+1\right) \right] }\left( ab+4\right) \mathcal{F}_{m+n}\left( a,b\right) $. So, from Theorem \ref{oz1} and $(i)$, we have
\begin{eqnarray*}
\mathcal{L}_{m}\left( a,b\right) \mathcal{L}_{n}\left( a,b\right) &=&\frac{a%
}{b}\left( \mathcal{F}_{m+1}\left( a,b\right) +\mathcal{F}_{m-1}\left(
a,b\right) \right) \frac{a}{b}\left( \mathcal{F}_{n+1}\left( a,b\right) +%
\mathcal{F}_{n-1}\left( a,b\right) \right) \\
&=&\frac{a^{2}}{b^{2}}\left\{ 
\begin{array}{c}
\left( \frac{b}{a}\right) ^{\varepsilon \left[ \left( m+1\right) \left(
n+1\right) \right] }\mathcal{F}_{m+n+2}\left( a,b\right)\\+\left( \frac{b}{a}%
\right) ^{\varepsilon \left[ \left( m+1\right) \left( n-1\right) \right] }%
\mathcal{F}_{m+n}\left( a,b\right)\\+\left( \frac{b}{a}\right) ^{\varepsilon \left[ \left( m-1\right) \left(n+1\right) \right] }\mathcal{F}_{m+n}\left( a,b\right)\\+\left( \frac{b}{a}\right) ^{\varepsilon \left[ \left( m-1\right) \left( n-1\right) \right] }%
\mathcal{F}_{m+n-2}\left( a,b\right)
\end{array}
\right\} 
\end{eqnarray*}
\begin{eqnarray*}
\mathcal{L}_{m}\left( a,b\right) \mathcal{L}_{n}\left( a,b\right)&=&\left( \frac{a}{b}\right) ^{2-\left[ \varepsilon \left( m+1\right)
\varepsilon \left( n+1\right) \right] }\left[ \mathcal{F}_{m+n+2}\left(
a,b\right) +2\mathcal{F}_{m+n}\left( a,b\right) +\mathcal{F}_{m+n-2}\left(
a,b\right) \right] \\
&=&\left( \frac{a}{b}\right) ^{2-\left[ \varepsilon \left( m+1\right)
\varepsilon \left( n+1\right) \right] }\left[ \left( \frac{b}{a}\right) 
\mathcal{L}_{m+n+1}\left( a,b\right) +\left( \frac{b}{a}\right) \mathcal{L}%
_{m+n-1}\left( a,b\right) \right] \\
&=&\left( \frac{a%
}{b}\right) ^{2-\left[ \varepsilon \left( m+1\right) \varepsilon \left(
n+1\right) \right] }\left( ab+4\right) \mathcal{F}_{m+n}\left( a,b\right)\,.
\end{eqnarray*}
\end{itemize}
\end{proof}

\begin{theorem}
\label{oz4} For $m,n,r$ $\in\mathbb{N}$ and $n\geq r$, the following equalities are hold:
\begin{itemize}
\item[$\boldsymbol{(i)}$] \ \ $\mathcal{F}_{n}^{m}\left( a,b\right) =\left( \frac{b}{a}%
\right) ^{\left\lfloor \frac{m}{2}\right\rfloor \varepsilon \left( n\right) }%
\mathcal{F}_{mn}\left( a,b\right)$\,,

\item[$\boldsymbol{(ii)}$] \ \ $\mathcal{F}_{n+1}^{m}\left( a,b\right) =\left( \frac{a}{b}%
\right) ^{\left\lfloor \frac{m+1}{2}\right\rfloor \varepsilon \left(
n\right) }\mathcal{F}_{1}^{m}\left( a,b\right) \mathcal{F}_{mn}\left(
a,b\right)$\,,

\item[$\boldsymbol{(iii)}$] \ $\ \mathcal{F}_{n-r}\left( a,b\right) \mathcal{F}%
_{n+r}\left( a,b\right) =\left( \frac{b}{a}\right) ^{\varepsilon \left(
n-r\right) }\mathcal{F}_{2}^{n}\left( a,b\right) =\left( \frac{b}{a}\right)
^{\left( -1\right) ^{n}\varepsilon \left( r\right) }\mathcal{F}%
_{n}^{2}\left( a,b\right)$\,,

\item[$\boldsymbol{(iv)}$] \ \ $\mathcal{L}_{n-r}\left( a,b\right) \mathcal{L}_{n+r}\left(
a,b\right) =\left( \frac{a}{b}\right) ^{\left( -1\right) ^{n}\varepsilon
\left( r\right) }\mathcal{L}_{n}^{2}\left( a,b\right)$\,,

\item[$\boldsymbol{(v)}$] \ \ $\mathcal{L}_{0}^{m}\left( a,b\right) \mathcal{F}%
_{mn}\left( a,b\right) =\left( \frac{b}{a}\right) ^{\left\lfloor \frac{m+1}{2%
}\right\rfloor \varepsilon \left( n\right) }\mathcal{L}_{n}^{m}\left(
a,b\right)$\,.
\end{itemize}
\end{theorem}

\begin{proof}
\begin{itemize}
\item[$\boldsymbol{(i)}$] We actually can write $\mathcal{F}_{n}^{m}\left( a,b\right) =
\mathcal{F}_{n}\left( a,b\right) \mathcal{F}_{n}\left( a,b\right) \cdots 
\mathcal{F}_{n}\left( a,b\right) $ ($m$-times). Now, by the condition $(i)$ of Theorem \ref{oz3}, we clearly obtain
\begin{eqnarray*}
\mathcal{F}_{n}^{m}\left( a,b\right) &=&\left\{ 
\begin{array}{c}
\left( \left( \frac{b}{a}\right) ^{\varepsilon (n)}\right) ^{\frac{m}{2}}
\mathcal{F}_{mn}\left( a,b\right) ,\ \ \ \ \ \  \ \ \ \ \ \ \text{ }m\ \text{even}\ \ \ \ \ \\ 
\left( \left( \frac{b}{a}\right) ^{\varepsilon (n)}\right) ^{\frac{m-1}{2}}
\mathcal{F}_{n\left( m-1\right) }\left( a,b\right) \mathcal{F}_{n}\left(
a,b\right) ,\text{ }m \ \text{odd}
\end{array}
\right. \\
&=&\left\{ 
\begin{array}{c}
\left( \left( \frac{b}{a}\right) ^{\varepsilon (n)}\right) ^{\frac{m}{2}}%
\mathcal{F}_{mn}\left( a,b\right) ,\ \ \text{ }m \ \text{even} \\ 
\left( \left( \frac{b}{a}\right) ^{\varepsilon (n)}\right) ^{\frac{m-1}{2}}%
\mathcal{F}_{mn}\left( a,b\right) ,\ \text{ }m \ \text{odd}%
\end{array}%
\right. \\
&=&\left( \frac{b}{a}\right) ^{\left\lfloor \frac{m}{2}\right\rfloor
\varepsilon \left( n\right) }\mathcal{F}_{mn}\left( a,b\right)\,.
\end{eqnarray*}

\item[$\boldsymbol{(ii)}$] Let us consider the left-hand side of the equality. As a
similar approximation in $(i)$, we write

\begin{eqnarray*}
\mathcal{F}_{n+1}^{m}\left( a,b\right) &=&\left\{ 
\begin{array}{c}
\left( \left( \frac{b}{a}\right) ^{\varepsilon (n+1)}\right) ^{\frac{m}{2}}%
\mathcal{F}_{m\left( n+1\right) }\left( a,b\right) ,\text{ }m\text{\ even } \\ 
\left( \left( \frac{b}{a}\right) ^{\varepsilon (n+1)}\right) ^{\frac{m-1}{2}}%
\mathcal{F}_{\left( n+1\right) \left( m-1\right) }\left( a,b\right) \mathcal{%
F}_{n+1}\left( a,b\right) ,\text{ }m\text{\ odd }%
\end{array}%
\right. 
\end{eqnarray*}
\begin{eqnarray*}
\mathcal{F}_{n+1}^{m}\left( a,b\right) &=&\left\{ 
\begin{array}{c}
\left( \left( \frac{b}{a}\right) ^{\varepsilon (n+1)}\right) ^{\frac{m}{2}}%
\mathcal{F}_{m\left( n+1\right) }\left( a,b\right) ,\text{ }m\text{\ even } \\ 
\left( \left( \frac{b}{a}\right) ^{\varepsilon (n+1)}\right) ^{\frac{m-1}{2}}%
\mathcal{F}_{m\left( n+1\right) }\left( a,b\right) ,\text{ }m\text{\ odd }%
\end{array}%
\right. \\
&=&\left( \frac{b}{a}\right) ^{\left\lfloor \frac{m}{2}\right\rfloor
\varepsilon \left( n+1\right) }\mathcal{F}_{mn+m}\left( a,b\right)\\
&=&\left( 
\frac{b}{a}\right) ^{\left\lfloor \frac{m}{2}\right\rfloor \varepsilon
\left( n+1\right) }\left( \frac{a}{b}\right) ^{\varepsilon (mn)}\mathcal{F}%
_{m}\left( a,b\right) \mathcal{F}_{mn}\left( a,b\right) \\
&=&\left( \frac{b}{a}\right) ^{\left\lfloor \frac{m}{2}\right\rfloor
\varepsilon \left( n+1\right) -\varepsilon (mn)}\mathcal{F}_{mn}\left(
a,b\right) \mathcal{F}_{m}\left( a,b\right)\,.
\end{eqnarray*}
Similarly, we can write $\mathcal{F}_{m}\left( a,b\right) =\left( \frac{a}{b}%
\right) ^{\varepsilon (m-1)}\mathcal{F}_{m-1}\left( a,b\right) \mathcal{F}%
_{1}\left( a,b\right) $. By iterative processes, we obtain $\mathcal{F}%
_{m}\left( a,b\right) =\left( \frac{a}{b}\right) ^{\left\lfloor \frac{m}{2}%
\right\rfloor }\mathcal{F}_{1}^{m}\left( a,b\right)$. Thus,
\begin{eqnarray*}
\mathcal{F}_{n+1}^{m}\left( a,b\right) &=&\left( \frac{b}{a}\right)
^{\left\lfloor \frac{m}{2}\right\rfloor \varepsilon \left( n+1\right)
-\varepsilon (mn)-\left\lfloor \frac{m}{2}\right\rfloor }\mathcal{F}%
_{mn}\left( a,b\right) \mathcal{F}_{1}^{m}\left( a,b\right) \\
&=&\left( \frac{a}{b}\right) ^{\left\lfloor \frac{m+1}{2}\right\rfloor
\varepsilon \left( n\right) }\mathcal{F}_{1}^{m}\left( a,b\right) \mathcal{F}%
_{mn}\left( a,b\right)\,.
\end{eqnarray*}

\item[$\boldsymbol{(iii)}$] From Theorem \ref{oz3} and $(i)$, we write
\begin{eqnarray*}
\mathcal{F}_{n-r}\left( a,b\right) \mathcal{F}_{n+r}\left( a,b\right)
&=&\left( \frac{b}{a}\right) ^{\varepsilon \left[ \left( n-r\right) \left(
n+r\right) \right] }\mathcal{F}_{2n}\left( a,b\right) \\
&=&\left( \frac{b}{a}\right) ^{\varepsilon \left[ \left( n-r\right) \left(
n+r\right) \right] }\mathcal{F}_{2}^{n}\left( a,b\right) \\
&=&\left( \frac{b}{a}\right) ^{\varepsilon \left( n-r\right) }\mathcal{F}%
_{2}^{n}\left( a,b\right)\,.
\end{eqnarray*}
Also, we give
\begin{eqnarray*}
\mathcal{F}_{n-r}\left( a,b\right) \mathcal{F}_{n+r}\left( a,b\right)
&=&\left( \frac{b}{a}\right) ^{\varepsilon \left[ \left( n-r\right) \left(
n+r\right) \right] }\mathcal{F}_{2n}\left( a,b\right) \\
&=&\left( \frac{b}{a}%
\right) ^{\varepsilon \left[ \left( n-r\right) \left( n+r\right) \right]
}\left( \frac{a}{b}\right) ^{\varepsilon \left( n\right) }\mathcal{F}%
_{n}^{2}\left( a,b\right) \\
&=&\left( \frac{b}{a}\right) ^{\left( -1\right)
^{n}\varepsilon \left( r\right) }\mathcal{F}_{n}^{2}\left( a,b\right)\,.
\end{eqnarray*}
\end{itemize}
\end{proof}

\section*{Conclusion}
\qquad
In this paper, we define the bi-periodic Lucas matrix sequence and give some
properties of this new sequence. Thus, it is obtained a new genaralization
for the matrix sequences and number sequences that have the similar
recurrence relation in the literature. By taking into account this
generalized matrix sequence and its properties, it also can be obtained
properties of the bi-periodic Lucas numbers. That is, if we compare the $1$%
\textit{st} row and $2$\textit{nd} column entries of obtained equalities for
matrix sequence in Section 2, we can get some properties for bi-periodic
Lucas numbers. Also, comparing the row and column entries of obtained
expressions for matrix sequences in Section 3, we can obtain relationships
between the bi-periodic Fibonacci and bi-periodic Lucas numbers. Finally,
some well-known matrix sequences as Lucas, $k$-Lucas and Pell-Lucas are special
cases of \{$\mathcal{L}_{n}\left( a,b\right) $\} matrix sequence. That is,
if we choose the different values of $a$ and $b$, then we obtain the
summations, generating functions, Binet formulas and relationships of the well-known matrix sequences in the literature:

\begin{itemize}
\item If we replace $a=b=1$ in $\mathcal{L}_{n}\left( a,b\right) $, we
obtain for Lucas matrix sequence.

\item If we replace $a=b=k$ in $\mathcal{L}_{n}\left( a,b\right) $, we
obtain for $k$-Lucas matrix sequence.
\end{itemize}

\end{document}